\documentclass[12pt]{amsart}%
\usepackage{amssymb,amsfonts}
\usepackage{amscd}
\usepackage{amsmath, mathtools}
\usepackage{amsfonts}
\usepackage{amssymb}
\usepackage{verbatim}
\usepackage{xcolor}
\usepackage{bm, enumerate}
\usepackage[normalem]{ulem}

\newtheorem{theorem}{Theorem}[section]
\newtheorem{lemma}[theorem]{Lemma}
\newtheorem{proposition}[theorem]{Proposition}

\newtheorem{example}[theorem]{Example}

\newtheorem{corollary}[theorem]{Corollary}
\newtheorem{remark}[theorem]{Remark}

\theoremstyle{plain}

\newtheorem{question}[theorem]{Question}

\newtheorem{case-new}{Case}

\numberwithin{equation}{section}

\newcommand*{\Ge}{\geqslant}
\newcommand*{\Le}{\leqslant}

\usepackage{etoolbox}

\makeatletter
\@namedef{subjclassname@2020}{%
  \textup{2020} Mathematics Subject Classification}
\makeatother

\usepackage{mathrsfs}
\usepackage{epsfig}
\usepackage[active]{srcltx}

\newcommand{\ncom}{\newcommand}
\ncom{\bq}{\begin{equation}}
\ncom{\eq}{\end{equation}}
\ncom{\beqn}{\begin{eqnarray*}}
\ncom{\eeqn}{\end{eqnarray*}}
\ncom{\beq}{\begin{eqnarray}}
\ncom{\eeq}{\end{eqnarray}}
\ncom{\nno}{\nonumber}
\ncom{\rar}{\rightarrow}
\ncom{\Rar}{\Rightarrow}
\ncom{\noin}{\noindent}
\ncom{\bc}{\begin{centre}}
\ncom{\ec}{\end{centre}}
\ncom{\sz}{\scriptsize}
\ncom{\rf}{\ref}
\ncom{\sgm}{\sigma}
\ncom{\Sgm}{\Sigma}
\ncom{\dt}{\delta}
\ncom{\Dt}{Delta}
\ncom{\lmd}{\lambda}
\ncom{\Lmd}{\Lambda}
\ncom{\eps}{\epsilon}
\ncom{\pcc}{\stackrel{P}{>}}
\ncom{\dist}{{\rm\,dist}}
\ncom{\sspan}{{\rm\,span}}
\ncom{\im}{{\rm Im\,}}
\ncom{\sgn}{{\rm sgn\,}}
\ncom{\ba}{\begin{array}}
\ncom{\ea}{\end{array}}
\ncom{\eop}{\hfill{{\rule{2.5mm}{2.5mm}}}}
\ncom{\eoe}{\hfill{{\rule{1.5mm}{1.5mm}}}}
\ncom{\eof}{\hfill{{\rule{1.5mm}{1.5mm}}}}
\ncom{\hone}{\mbox{\hspace{1em}}}
\ncom{\htwo}{\mbox{\hspace{2em}}}
\ncom{\hthree}{\mbox{\hspace{3em}}}
\ncom{\hfour}{\mbox{\hspace{4em}}}
\ncom{\hsev}{\mbox{\hspace{7em}}}
\ncom{\vone}{\vskip 2ex}
\ncom{\vtwo}{\vskip 4ex}
\ncom{\vonee}{\vskip 1.5ex}
\ncom{\vthree}{\vskip 6ex}
\ncom{\vfour}{\vspace*{8ex}}
\ncom{\norm}{\|\;\;\|}
\ncom{\integ}[4]{\int_{#1}^{#2}\,{#3}\,d{#4}}
\ncom{\inp}[2]{\langle{#1},\,{#2} \rangle}
\ncom{\Inp}[2]{\Langle{#1},\,{#2} \Langle}
\ncom{\vspan}[1]{{{\rm\,span}\#1 \}}}
\ncom{\dm}[1]{\displaystyle {#1}}

\usepackage[pagewise]{lineno}



\begin{document}
\title[Completely alternating sequences]{A characterization of completely alternating functions}

\dedicatory{Dedicated to Professor Gadadhar Misra on the occasion of his 65th birthday}

 \author[Bhattacharjee]{Monojit Bhattacharjee}
 \address{Department of Mathematics, Birla Institute of Technology and Science, Pilani, K K Birla Goa Campus, Zuarinagar, Sancoale, Goa 403726, India}

 \email{monojitb@goa.bits-pilani.ac.in, monojit.hcu@gmail.com}

\author[R. Nailwal]{Rajkamal Nailwal}
\address{Institute of Mathematics, Physics and Mechanics, Ljubljana, Slovenia.}
\email{rajkamal.nailwal@imfm.si, raj1994nailwal@gmail.com}
\begin{abstract}
In this article, we characterize completely alternating functions on an abelian semigroup $S$ in terms of completely monotone functions on the product semigroup $S\times \mathbb Z_+$. We also discuss completely alternating sequences induced by a class of rational functions and obtain a set of sufficient conditions (in terms of it's zeros and poles) to determine them. As an application, we show a complete characterization of several classes of completely monotone functions on $\mathbb Z_+^2$ induced by rational functions in two variables. We also derive a set of necessary conditions 
for the complete monotonicity of the sequence $\{\prod_{i=1}^{k}\frac{(n+a_i)}{(n+b_i)}\}_{n \in \mathbb Z_+}, a_i, b_i \in (0,\infty)$.

\end{abstract}

\maketitle

\begin{section}{Introduction}

Let $S$ be an abelian semigroup with identical involution and $\mathbb{R}^S$ denotes the set of real-valued functions on $S.$ For $a \in S,$ define $E_a$ the shift operator on $\mathbb{R}^S$ by $E_{a}f(s):=f(s+a),$  and the backward difference operator $\nabla_{a}:=I-E_a,$ where $I$ denotes the identity operator on $\mathbb R^S.$  A function $\phi: S\rightarrow \mathbb R$ is called completely monotone if it is nonnegative and if for all finite set $\{a_1, \ldots,a_n\}\subseteq S$ and $s\in S,$
\beqn
    \nabla_{a_1}\ldots\nabla_{a_n} \phi(s)\Ge 0.
    \eeqn
A function $\psi: S\rightarrow \mathbb R$ is called completely alternating if for all finite set $\{a_1, \ldots,a_n\}\subseteq S$ and $s\in S,$
\beqn
    \nabla_{a_1}\ldots\nabla_{a_n} \psi(s)\Le 0.
    \eeqn   
The set of completely monotone and completely alternating functions on $S$ is denoted by $\mathcal{M}(S)$ and $\mathcal{A}(S)$, respectively.

Given a set $X,$ let $X^n$ denote the Cartesian product of
$X$ with itself
$n$ times. Let $\mathbb{Z}_+ $ and $\mathbb{R}_+ $ denote the set of non-negative integers and the set of non-negative real numbers, respectively.
Below, we enlist well-known characterizations of {\it completely monotone} and {\it completely alternating} functions on the abelian semigroup $ \mathbb Z_+^n $, in terms of measure (see \cite[Propositions 6.11, 6.12, p. 134]{BCR1984}).
    \begin{theorem}\label{BCR}
        For $\phi:\mathbb Z_+^n\rightarrow \mathbb R,$ the following holds:
        \begin{enumerate}
            \item[$\mathrm{(i)}$] $\phi\in \mathcal{M}(\mathbb Z_+^n)$ if and only if it is a {\it Hausdorff moment net}, that is, there exists a positive Borel measure $\mu$ on $[0,1]^n$ such that
    \beq\label{rm}
    \phi(k)=\int_{[0,1]^n}t^kd\mu(t), \quad k \in \mathbb Z_+^n.
    \eeq
    \item[$\mathrm{(ii)}$]  $\phi \in \mathcal{A}(\mathbb Z_+^n)$ if and only if there exist $a\in \mathbb R, b \in \mathbb R_+^n$ and a positive Borel measure $\mu$ on $[0,1]^n\setminus \{\mathbf{1}\}$ such that
        \end{enumerate}
        \beq\label{CA-sequence}
    \phi(k)=a+\langle b,k \rangle+\int_{[0,1]^n\setminus \{\mathbf{1}\}}(1-t^k)d\mu(t), \quad k \in \mathbb Z_+^n,
    \eeq
    where $\langle \cdot,\cdot \rangle$ represents the standard inner product on $\mathbb R^n.$ 
    \end{theorem}
 
   Note that a completely alternating function is invariant under a shift by a constant. Due to this fact, in this article, we work with completely alternating functions which are positive.
  
Recall that the measure in \eqref{rm} is called the {\it representing measure} of $\phi$ and it is unique. This is a consequence of the $n$-dimensional Weierstrass theorem and the Riesz representation theorem.   In \cite{FH1923}, F. Hausdorff characterized a completely monotone sequence by Hausdorff moment sequence and the above characterization about completely alternating sequence was first obtained as a special case of the Levy-Khinchin representation theory on abelian semigroups (see \cite{BCR1976}, \cite{RRZ2012}). For a comprehensive study on these topics, we refer the reader to \cite{BCR1976}, \cite{BCR1984}, \cite{RRZ2012}, \cite{W1941}. 
  
A lot of work has been done on the characterization of completely monotone functions in terms of Hausdorff moment sequences, for instance, \cite{BD2005, SS2019, RZ2019} and references therein. However, the characterization of a completely monotone sequence induced by rational functions in terms of zeroes and poles is challenging and seems to be beyond the limit of present understanding. One of the important results in this direction is the sufficient conditions to identify completely monotone sequences introduced by Ball \cite{Kball} which are solely dependent on the zeroes and poles of the rational functions (see Theorem~\ref{ball-result}). This article aims to determine certain {\it completely alternating} sequences interpolated by rational functions with real zeros and poles in terms of sufficient conditions. In this article, we deal with the following class of rational functions, that is, for $k\in \mathbb N$  
    
    \beq \label{Problem-1}
    r_{k}(x):=\frac{p(x)}{\prod_{j=1}^{k}(x+b_i)},
    \eeq
    where $p$ is a polynomial of degree at most $k+1$ and $0<b_1<b_2<\ldots<b_k.$

    Here, more precisely, we concentrate on the following moment  problem$:$
    \begin{question} 
   Among all the sequences $\{r_k(n)\}_{n \in \mathbb Z_+}$ of rational functions $r_{k}(x)$, defined in \eqref{Problem-1}, which are {\it completely alternating}?
     \end{question}
 From the operator theoretic aspect, completely monotone and completely alternating sequences are used in revealing the connection with subnormal and completely hyperexpansive operators, respectively (see \cite{AC2017}, \cite{AC17}, \cite{ACJS}, \cite{ACN}, \cite{AA1996}, \cite{AR2002}, \cite{BCE2022}). 

Before proceeding further, we encounter a non-example which shows that any sequence of rational functions $\{r_{k}(n)\}_{n \in \mathbb Z_+}$ is not {\it completely alternating} sequence.
\begin{example}
Let $r(x)=\frac{x+6}{x+5}.$ For the semigroup $\mathbb {Z}_+,$ we write $\nabla=\nabla_1.$ Note that
    \beqn
    \nabla r(n)=\frac{n+6}{n+5}-\frac{n+7}{n+6}, \quad n \in \mathbb Z_+.
    \eeqn
    Since $ \nabla r(n)|_{n=0}=0.033,$ $\{r(n)\}_{n \in \mathbb Z_+}$ is not a completely alternating sequence.
\end{example}
The above example motivates us to identify completely alternating sequences among the sequence of rational functions of the form \eqref{Problem-1}. 

We now state the main results of this paper. The first one characterizes a completely alternating function on an abelian semigroup with the identical involution, and the second result involves a set of conditions sufficient to determine the completely alternating sequences induced by rational functions.
 \begin{theorem}\label{CAJCM}
        Let S be an abelian semigroup with the identical involution and let $\psi : S \rightarrow \mathbb R$ be a
strictly positive function. The function $f : S \times \mathbb Z_+ \rightarrow \mathbb R $ defined by 
$$f(s,m)=\frac{1}{\psi(s)+m}, \quad s \in S, m \in \mathbb Z_+$$
is completely monotone on the product semigroup $S\times \mathbb Z_+$ with the identical involution if and only if $\psi$ is completely alternating on $S$ i.e.

$$f\in \mathcal{M}(S\times \mathbb Z_+) \Leftrightarrow  \psi \in \mathcal{A}(S).$$

    \end{theorem}
 \begin{theorem}\label{main-3}
        Let $k\in \mathbb N$ and $\{b_i\}_{i=1}^k$ be a strictly increasing sequence of positive real numbers. Consider the polynomial $q$ given by $q(x)=(x+b_1)(x+b_2)\ldots(x+b_k).$ Let $p$ be any polynomial of degree at most $k+1$ with the coefficient of $x^{k+1}$ non-negative. Then, the sequence $\left\{\frac{p(n)}{q(n)}\right\}_{n \in \mathbb Z_+}$ is completely alternating provided 
        \beqn
        \sum_{i=1}^lc_i\Le0, \quad l \in \{1,\ldots, k\},
        \eeqn
         where $c_i=\frac{p(-b_i)}{\prod_{j=1, j \neq i}^k(b_j-b_i)}, \, i \in \{1,\ldots, k\}.$
    \end{theorem}
   
    As an application of this result, we also develop a set of conditions sufficient to identify the completely alternating sequences induced by rational functions defined in \eqref{Problem-1}, with $p(x)= \prod_{i=1}^l (x + a_i)$ and $l=k$ or $k+1$, which solely depend on its zeros and poles (see Corollary \ref{CA-Ex}, also see Theorem~\ref{CA-root-2}). These sufficient conditions are very easy to check. They turn out to be necessary in a sense discussed in Remark~\ref{Necessray}. 
    In a recent article \cite{ACN}, the authors have studied complete monotonicity of the net $\left\{\frac{1}{p(m)+q(m)n}\right\}_{m,n \in \mathbb Z_+}$ (see \cite[Theorem~1.4]{ACN}) and solved the Cauchy dual subnormality problem for a class of toral $3$-isometric weighted $2$-shifts which are separate $2$-isometries. Combining Theorem~\ref{CAJCM} and \ref{CA-root-2}, we also obtain a complete characterization of several classes of completely monotone nets induced by rational functions in two variables (see Corollary~\ref{bi-comp}). Also, we recover Theorem 3.1 in \cite{ACN} with another characterization of completely monotone sequences as an application of Theorem~\ref{CAJCM} (see Corollary~\ref{bi-poly}).   
\end{section}

Next we turn our attention to the complete monotonicity of the following class of rational function
 \beq \label{Problem-2}
    f_{k}(x)=\prod_{i=1}^{k}\frac{(x+a_i)}{(x+b_i)},
    \eeq
    where $0<a_1\Le a_2\Le \ldots \Le a_k$ and $0<b_1\Le b_2\Le\ldots\Le b_k.$   In \cite{Kball}, K. Ball obtain a set of sufficient conditions in terms of zeros and poles of \eqref{Problem-2} for which \eqref{Problem-2} is completely monotone.  For the reader's convenience, we recall the result here:
    \begin{theorem}[K. Ball, 1994] \label{ball-result}
        The function in \eqref{Problem-2} is completely monotone provided 
        \beq \label{sc-ball}
        \sum_{i=1}^{l}b_i\Le \sum_{i=1}^{l}a_i,
        \eeq
        for every $l \in \{1, \ldots, k\}.$
    \end{theorem}

     To the best of our knowledge, it was not known whether the Ball's sufficiency conditions for the complete monotonicity of \eqref{Problem-2} are necessary. A possible reason could be that these conditions are both necessary and sufficient for $k=1$ and $k=2$ (to be seen in  Section \ref{sec Necessary conditions}). It was mentioned in his paper, that those conditions \eqref{sc-ball} are almost necessary (see \cite[Page 3]{Kball}) that is, if each $f_j, j=1, \ldots k,$ is completely monotone then \eqref{sc-ball} hold. This follows from \cite[Remark]{Kball}. In Section \ref{sec Necessary conditions}, we prove that Ball's sufficiency conditions are not necessary and obtain a new set of necessary conditions for this class which is a generatization of Ball's sufficiency condition. Before we state this result, we need the following notation. Let $S_k$ denote the set of all permutations on the set $\{1, \ldots,k\}$.
 \begin{theorem}\label{main-1}
  Let $\{a_i\}_{i=1}^k$ and $\{b_i\}_{i=1}^k$ be two non-decreasing sequences of positive real numbers. Let $p(x)=\prod_{i=1}^{k}(x+a_i)$ and $q(x)=\prod_{i=1}^{k}(x+b_i).$ Then
  \begin{enumerate}
      \item   If the sequence  $\left\{\frac{p(n)}{q(n)}\right\}_{n \in \mathbb Z_+}$ is  completely monotone then there exists a $\sigma\in S_{k}$ with $\sigma(1)=1$ such that  
 \beq \label{sigma-tau}
 \sum_{i=1}^{l}b_{\sigma(i)}\Le \sum_{i=1}^{l}a_{\sigma(i)}
 \eeq
 for every $l \in \{1,\ldots, k\}.$
 \item  If the sequence  $\left\{\frac{p(n)}{q(n)}\right\}_{n \in \mathbb Z_+}$ is  completely alternating then there exists a $\sigma\in S_{k}$ with $\sigma(1)=1$ such that 
 \beqn \label{sigma-tau-1}
\sum_{i=1}^{l}a_{\sigma(i)}\Le\sum_{i=1}^{l}b_{\sigma(i)}
 \eeqn
 for every $l \in \{1,\ldots, k\}.$
  \end{enumerate}
  
 \end{theorem}

    \begin{section}{A proof the Theorem~\ref{CAJCM}}
    In this section, we present a proof of Theorem~\ref{CAJCM}. In \cite[Proposition~6]{AA1996} (also see \cite[Theorem~2.5]{SS2019}), the authors have shown a correspondence between completely alternating sequences and completely monotone sequences. As far as we know, there are only a few interesting examples of complete monotone functions in more than one variable in the literature. Our result in this section allows us to construct non-trivial examples from completely alternating functions (see Corollary~\ref{bi-comp}).

    For the reader's convenience, we recall the following result which we need in the proof of  Theorem~\ref{CAJCM} (\cite[Proposition~6.10, p. 133]{BCR1984}).

    \begin{proposition}\label{CA-Exp}
        Let $\psi:S \rightarrow \mathbb R.$ Then $\psi \in \mathcal{A}(S)$ if and only if $e^{-t\psi}\in \mathcal{M}(S)$ for all $t>0.$
    \end{proposition}
   We are now ready to prove Theorem~\ref{CAJCM}.
    \begin{proof}[Proof of Theorem~\ref{CAJCM}]
        To prove the equivalence, we note the following identity: 
    \beq \label{JCM-eq}
    f(s,m)=\frac{1}{\psi(s)+m}=\int_{(0,1)}t^mt^{\psi(s)-1}dt, \quad s \in S, m \in \mathbb Z_+.
    \eeq
    
    $\Leftarrow$: Assume that $\psi \in \mathcal{A(S)}$. By Proposition~\ref{CA-Exp},  $\{e^{(\ln t)\psi}\}\in \mathcal{M}(S)$ for every $t \in (0,1).$ Equivalently, $t^{\psi}\in \mathcal{M}(S)$  for every $t \in (0,1).$ This together with the fact that $\{t^m\}_{m \in \mathbb Z_+}, t \in (0,1),$ is completely monotone yields that $t^mt^{\psi-1} \in \mathcal M( {S\times \mathbb Z_+})$ for every $t\in (0,1).$
       
    Combining this with \eqref{JCM-eq}, we get
\beqn
\nabla_{a_1}\ldots\nabla_{a_k}f(s,m)= \int_{(0,1)}\nabla_{a_1}\ldots\nabla_{a_k}t^mt^{\psi(s)-1}dt\Ge 0,\quad a_i \in S\times \mathbb Z_+, s\in S, m \in \mathbb Z_+.
\eeqn
    This shows that $f\in \mathcal{M}(S\times \mathbb Z_+).$ 
    
    $\Rightarrow$: Assume that $f\in \mathcal{M}(S\times \mathbb Z_+).$  This together with \eqref{JCM-eq} yields that for every  $a_i \in S\times \mathbb Z_+, s\in S,$
    \beq \label{integal}
   \nabla_{a_1}\ldots\nabla_{a_k}f(s,m)= \int_{(0,1)}\nabla_{a_1}\ldots\nabla_{a_k}t^mt^{\psi(s)-1}dt\Ge 0,\quad m \in \mathbb Z_+.
    \eeq
    Since $f \in \mathcal{M}(S\times \mathbb Z_+),$ $$\nabla_{b_1}\ldots\nabla_{b_k}f(s,\cdot) \in \mathcal{M}(\mathbb Z_{+})$$   for every finite set $\{b_1, \ldots,b_k\}\subseteq S,s \in S.$  Note that for a finite set $\{b_1, \ldots,b_k\}\subseteq S,$ the function $w(t):=\prod_{i=1}^k\nabla_{b_{i}} t^{\psi(s)-1} \in L^{1}[0,1].$ This can be seen by taking $m=0$ in \eqref{integal}.  Therefore the representing measure of the sequence $\{\nabla_{b_1}\ldots\nabla_{b_k}f(s,m)\}_{m \in \mathbb Z_+}$ is a weighted Lebesgue measure (by \eqref{integal}) with weight function  $w(t)$ is positive for almost every $t \in (0,1).$ But since $w(\cdot)$ is continuous on $(0,1),$ we have
    \beqn
    \prod_{i=1}^k\nabla_{b_{i}} t^{\psi(s)-1}\Ge 0,\quad b_i\in S, k \in \mathbb Z_+, \, t \in (0,1).
    \eeqn
    This yields that for every $t \in (0,1),$ $t^{\psi} \in \mathcal{M(S)}$ and equivalently, $\psi \in \mathcal{A}(S).$
      \end{proof}
      The following result is a consequence of Theorem \ref{CAJCM}.
      \begin{corollary}
          Assume the notations of Theorem \ref{CAJCM} and $\alpha>0.$ Then $f(s,\alpha m)\in \mathcal{M}(S\times \mathbb Z_+)$ if and only if $ \psi \in \mathcal{A}(S). $ 
      \end{corollary}
\begin{proof}
    This follows from Theorem \ref{CAJCM} and the fact that multiplying by a positive real number does not change the completely alternating (resp. completely monotone) property of a function. 
\end{proof}
      
\end{section}

\begin{section}{completely alternating sequences induced by rational functions}
In this section, we present a proof of Theorem~\ref{main-3}. Along with that, we show that every member of a class of sequences induced by rational functions is completely alternating if it satisfies a set of conditions developed solely in terms of its zeros and poles.   

Note that if the measure $\mu$ of the completely alternating sequence $\{a(k)\}_{k \in \mathbb Z_+}$ in \eqref{CA-sequence} is finite then  \eqref{CA-sequence} can be rewritten as 
\beq \label{CA-measure}
   a(k)=a+ bk -\int_{[0,1)}t^kd\mu(t), \quad k \in \mathbb Z_+,
   \eeq
   for some $a\in \mathbb R$ and $b\in \mathbb R_+.$  Conversely, if $a(k)=\frac{p(k)}{q(k)}$ where $p$ and $q$ are polynomials with  $q(x)=\prod_{i=1}^k(x+b_i), b_i>0$ and $\deg p \Le \deg q+1$ then $a(k)$ can be written as in \eqref{CA-measure} (using partial fraction decomposition) with $\mu$ being a finite signed measure (since $a(0)$ is finite). For irreducible 
$q$ with zeros in the left half-plane, one can use the method from \cite[p. 800]{AC2017} to derive an expression like \eqref{CA-measure}. We will make use of this expression throughout this section.

       We begin with the following result, which deals with sequences generated by rational functions of at most degree 2 and finds out when it will be a completely alternating sequence in terms of necessary and sufficient conditions. It has a deep impact on identifying a class of completely alternating sequences.
  \begin{theorem}\label{CA-root-2}
      Let $a_1,a_2,b_1,b_2 \in \mathbb R$ with $a_1\Le a_2$ and $0<b_1\Le b_2$, then we have the following:
         \begin{itemize} \label{CA-root}
        \item[$\mathrm{(i)}$] Let $r(x)=\frac{x+a_1}{x+b_1}.$ Then the sequence $\{r(n)\}_{n \in \mathbb Z_+}$ is completely alternating if and only if  $a_1\Le b_1.$ 
         \item[$\mathrm{(ii)}$] Let $r(x)=\frac{(x+a_1)(x+a_2)}{(x+b_1)(x+b_2)}.$ Then the sequence $\{r(n)\}_{n \in \mathbb Z_+}$ is completely alternating if and only if  $$a_1\Le b_1\Le a_2, a_1+a_2\Le b_1+b_2.$$ 
         \item[$\mathrm{(iii)}$] Let $r(x)=\frac{(x+a_1)(x+a_2)}{x+b_1}.$  Then the sequence  $\{r(n)\}_{n \in \mathbb Z_+}$ is completely alternating if and only if  $a_1\Le b_1\Le a_2.$
         \end{itemize}
      \end{theorem}
      \begin{proof}
        $\mathrm{(i)}$ This can be seen from the following identity:
          $$r(n)=1+\frac{a_1-b_1}{n+b_1}=1+\int_{(0,1)}t^{n}(a_1-b_1)t^{b_1-1}dt, \quad n \in \mathbb Z_+.$$\\
         $\mathrm{(ii)}$ We divide the proof into two cases. Assume that $b_1< b_2.$ By the partial fraction decomposition, 
          \beqn
          r(x)=1+\frac{c_1}{(x+b_1)}+\frac{c_2}{(x+b_2)},
          \eeqn
          where $c_1=\frac{(-b_1+a_1)(-b_1+a_2)}{b_2-b_1}$ and $c_2=\frac{(-b_2+a_1)(-b_2+a_2)}{b_1-b_2},$ we note that
          \beqn
r(n)=1+\int_{(0,1)}t^n(c_1t^{b_1}+c_2t^{b_2})t^{-1}dt, \quad n \in \mathbb Z_+.
          \eeqn
          Since $r(0)$ is finite, the function $(c_1t^{b_1}+c_2t^{b_2})t^{-1}$ is Lebesgue integrable on $[0,1].$
        In view of Theorem~\ref{BCR}(ii), it suffices to check that $c_1t^{b_1}+c_2t^{b_2}\Le0$ for all $t \in (0,1)$ to establish both necessity and sufficiency. This condition is equivalent to requiring $c_1\Le0$ and $c_1+c_2\Le0.$ Further, it is equivalent to $a_1\Le b_1\Le a_2, a_1+a_2\Le b_1+b_2.$ 
            Now if $b_1= b_2$ then again by the partial fraction decomposition, 
          \beqn
          r(x)=1+\frac{c_1}{(x+b_1)}+\frac{c_2}{(x+b_1)^2},
          \eeqn
          where $c_1=a_1+a_2-2b_1$ and $c_2=(-b_1+a_1)(-b_1+a_2).$ Note that 
          \beqn
r(n)=1+\int_{(0,1)}t^n(c_1-c_2 \ln(t))t^{b_1-1}dt, \quad n \in \mathbb Z_+.
          \eeqn
It is easy to verify that $c_1-c_2\ln(t)\Le 0, t \in (0,1)$ if and only if $c_1\Le0 $ and $ c_2\Le 0.$ Latter conditions are equivalent to $a_1+a_2\Le 2b_1$ and $b_1\in [a_1,a_2].$ This completes the proof.\\
          $\mathrm{(iii)}$  Analogous argument, as used in the second part, works here as well.
      \end{proof}
The following lemma plays a crucial role in proving Theorem~\ref{main-3}.
\begin{lemma}\label{imp-inq}
    Let $k \in \mathbb N.$ For any collection of real numbers $\{c_i\}_{i=1}^{k}$ satisfying  $ \sum_{i=1}^lc_i\Le0, \quad l \in \{1,\ldots, k\},$ and any collection $\{b_i\}_{i=1}^{k}$ of strictly increasing positive real numbers, 
        the function $$w(t)=\sum_{i=1}^kc_it^{b_i}\Le 0, \quad t \in (0,1).$$
\end{lemma}
\begin{proof} We use the induction on $k$ to prove the lemma.
     For $k=1,$ since $c_1 \Le 0,$ $w(t)=c_1t^{b_1}\Le 0, t\in (0,1).$ Assume the induction hypothesis i.e. the statement of the lemma is true for any $j < k.$ Now consider $w(t)=\sum_{i=1}^kc_it^{b_i}$ where $c_i$ and $b_i$ are as in the statement. Let $J=\{i: c_i>0\}.$ If $J$ is empty, then we are done. Let $i_0$ be the largest element of $J.$ Let if possible $k \notin J.$ Then $i_0\neq k.$ For every $t \in (0,1),$
    \beqn
    w(t)=\sum_{i=1}^{i_0-1}c_i t^{b_i}+c_{i_0}t^{b_{i_0}}+\sum_{ i =i_0+1 }^{k}c_it^{b_i}
     \Le\sum_{i=1}^{i_0-1}c_i t^{b_i}+c_{i_0}t^{b_{i_0}}.
    \eeqn
    By the induction hypothesis, we have 
    $$\sum_{i=1}^{i_0-1}c_i t^{b_i}+c_{i_0}t^{b_{i_0}}\Le 0, \quad t \in [0,1].$$
     If $k \in J.$ Then $i_0=k.$ We divide this case into the following two cases.\\
      \textbf{Case (i):  $c_{k-1}<0.$}
    Note that for every $t \in (0,1),$
    \beqn
    w(t)=\sum_{i=1}^{k-2}c_i t^{b_i}+c_{k-1}t^{b_{k-1}}+c_{k}t^{b_{k}}
     \Le\sum_{i=1}^{k-2}c_i t^{b_i}+(c_{k-1}+c_k)t^{b_{k}}
     \Le 0.
     \eeqn
     The last inequality follows from the induction hypothesis.\\
     \textbf{Case (ii): $c_{k-1}\Ge 0.$}
     Note that for every $t \in (0,1),$
    \beqn
    w(t)=\sum_{i=1}^{k-2}c_i t^{b_i}+c_{k-1}t^{b_{k-1}}+c_{k}t^{b_{k}}
     \Le\sum_{i=1}^{k-2}c_i t^{b_i}+(c_{k-1}+c_k)t^{b_{k-1}}
     \Le 0.
     \eeqn
     The last inequality follows from the induction hypothesis.
     This completes the proof.
\end{proof}
Next, we present the proof of Theorem~\ref{main-3} by generalizing the idea used in the proof of Theorem~\ref{CA-root}.
    
    \begin{proof}[Proof of Theorem~\ref{main-3}]
        Note that by partial fraction decomposition, we have
    \beq \label{coef-1}
    \frac{p(x)}{q(x)}=a_0+a_1x+\frac{c_1}{x+b_1}+\cdots+\frac{c_k}{x+b_k},
    \eeq
    where $c_i=\frac{p(-b_i)}{\prod_{j=1, j \neq i}^k(b_j-b_i)}$ and $a_1$ is the coefficient of $x^{k+1}$ in $p$ which is non negative. It is easy to see using \eqref{coef-1} that
    \beqn
    \frac{p(n)}{q(n)}=a_0+a_1n + \int_{(0,1)}t^nw(t)t^{-1}dt,\quad n \in \mathbb Z_+,
    \eeqn
    where $w(t)=\sum_{i=1}^kc_it^{b_i}, t \in (0,1).$ To prove the sequence  $\left\{\frac{p(n)}{q(n)}\right\}_{n \in \mathbb Z_+}$ is completely alternating, by \eqref{CA-measure}, it is now sufficient to check that $w(t)\Le 0,$ for $t \in (0,1).$ This follows from Lemma~\ref{imp-inq}.  This completes the proof.
    \end{proof}
 \begin{remark}
     In Theorem~\ref{main-3}, $\deg p \Le k+1$ turns out to be a necessary condition. Otherwise, if the sequence corresponding to the rational function $r_k$ with $\deg p > k+1$ is {\it completely alternating} then
     the sequence $\left\{\frac{p(n)}{n\prod_{i=1}^{k}(n+b_i)}\right\}_{n \in \mathbb N}$ turns out to be divergent, that is, 
     \[
     \lim_{n\to \infty}\frac{p(n)}{n\prod_{i=1}^{k}(n+b_i)}
     \] is not finite. But on the other hand, according to the Theorem~\ref{BCR}(ii), for a completely alternating sequence $\{a_n\}_{n \in \mathbb Z_+}$, ${\lim\limits_{n \to \infty }}\frac{a_n}{n}=b,$ where $b$ is a non negative real number as in \eqref{CA-sequence}.  This also explains that the non-negativity of the coefficient of $x^{k+1}$ in Theorem~\ref{main-3} is also necessary.
     \end{remark}
    
    As an application of Theorem~\ref{main-3}, we obtain the following classes of completely alternating sequences. 
    \begin{corollary}\label{CA-Ex}
         Let $q(x)=(x+b_1)\ldots(x+b_{k})$ and  $p$ be any of the following$:$
         \begin{enumerate}
             \item[$\mathrm{(i)}$]  $p(x)=(x+a_1)\ldots(x+a_{k})$ with  $0< a_1 < b_1 < a_2 < b_2 < \ldots < a_k < b_k,$ 
              \item[$\mathrm{(ii)}$] $p(x)=(x+a_1)\ldots(x+a_{k+1})$ with  $0< a_1 < b_1 < a_2 < b_2 < \ldots < a_k < b_k< a_{k+1}.$ 
         \end{enumerate}
         Then the sequence $\left\{\frac{p(n)}{q(n)}\right\}_{n \in \mathbb Z_+}$ is completely alternating.
    \end{corollary}
    \begin{proof}
        (i) Note that 
        \beqn
        \frac{p(x)}{q(x)}=1+\sum_{i=1}^{k}\frac{c_{i}}{(x+b_i)},
        \eeqn
        where $c_i= \frac{\prod_{l=1}^k (a_l-b_i)}{\prod_{1 \Le l \neq i \Le k}(b_l-b_i)}, \,i \in \{1, \ldots,k\}.$ This can be rewritten as
        \beqn
        c_i=- \frac{\prod_{1 \Le l \Le i} (b_i-a_{l})}{\prod_{1 \Le l \Le i-1}(b_i-b_l)}\frac{\prod_{i+1 \Le l \Le k} (a_l-b_i)}{\prod_{i+1 \Le l \Le k}(b_l-b_i)},
        \eeqn
        (we assume that the product over the empty set is $1$). Under the given conditions, it is easy to see that $c_i<0.$ By Theorem~\ref{main-3}, the sequence $\left\{\frac{p(n)}{q(n)}\right\}_{n \in \mathbb Z_+}$ is completely alternating.  A similar proof works for (ii).
    \end{proof}

    \begin{remark}\label{Necessray}
        The sufficient conditions in Corollary~\ref{CA-Ex} turn out to be necessary if we further assume that, in the first case where the degree of $p$ is same as the degree of $q$, the sequences $\left\{ \frac{(n+a_i)(n+a_{i+1})}{n+b_i} \right\}_{n \in \mathbb Z_+}$ for each $i \in \{1,\ldots,k-1\}$ and $\left\{ \frac{n+a_k}{n+b_k} \right\}_{n \in \mathbb Z_+}$ are complete alternating and in the second case where the degree of $p$ is bigger than the degree of $q$ by 1, the sequences $\left\{ \frac{(n+a_i)(n+a_{i+1})}{n+b_i} \right\}_{n \in \mathbb Z_+}$ for each $i \in \{1,\ldots,k\}$ are completely alternating. The proof follows from the Theorem~\ref{CA-root-2}.
    \end{remark}
The next result helps in finding more classes of completely alternating sequences where not all $c_i$'s are negative.
    \begin{proposition}\label{sumc-i}
        For any rational function $\prod_{i=1}^k\frac{x+a_i}{x+b_i},$ let $c_i$ denote the coefficient in the partial fraction decomposition of $r(x).$ Then $\sum_{i=1}^k c_i= \sum_{i=1}^k (a_i-b_i).$
    \end{proposition}
    \begin{proof}
        We use induction to prove the statement. For $k=1,$
        $$\frac{x+a_1}{x+b_1}=1+\frac{a_1-b_1}{x+b_1}.$$
        So the statement is true for $k=1,$ since $c_1=a_1-b_1.$ Assume that the statement is true for any $l<k.$ Then
        \beqn
\prod_{i=1}^{k}\frac{x+a_i}{x+b_i}&=&\left(\frac{x+a_1}{x+b_1}\right)\prod_{i=2}^k\left(\frac{x+a_i}{x+b_i}\right)=\left(1+\frac{a_1-b_1}{x+b_1}\right)\left(1+\sum_{i=2}^k\frac{c_i^{'}}{x+b_i}\right)\\
&=& 1+\frac{a_1-b_1}{x+b_1}+\sum_{i=2}^k\frac{c_i^{'}}{{x}+b_i}+\sum_{i=2}^k\frac{(a_1-b_1)c_i^{'}}{(x+b_1)(x+b_i)}\\
&=& 1+\frac{a_1-b_1}{x+b_1}+\sum_{i=2}^k\frac{c_i^{'}}{x+b_i}+\sum_{i=2}^k\frac{(a_1-b_1)c_i^{'}}{b_1-b_i}\left(\frac{1}{x+b_i}-\frac{1}{x+b_1}\right).
\eeqn
From this expression, we note that $\sum_{i=1}^kc_i=a_1-b_1+\sum_{i=2}^{k}c_{i}^{'}.$ From the induction hypothesis, we obtain $\sum_{i=1}^kc_i=\sum_{i=1}^{k}a_i-b_i.$
    \end{proof}
Now we present some classes of completely alternating sequences for which not all $c_i$'s are negative.
 \begin{corollary}\label{c_n-pos}
      Let $k\geq 2$ and $0<a_1<b_1<a_2<b_2<\ldots<a_{k-1}<b_{k-1}<b_k<a_k$ with $a_k\Le b_k+\sum_{i=1}^{k-1}(b_i-a_i).$ Then the sequence $\left\{\prod_{i=1}^k\frac{(n+a_i)}{(n+b_i)}\right\}_{n \in \mathbb Z_+}$ is completely alternating.  
    \end{corollary}
    \begin{proof}
    Let $c_i$'s be as in Theorem~\ref{main-3}.  Then 
     \beqn
        c_i=- \frac{\prod_{1 \Le l \Le i} (b_i-a_{l})}{\prod_{1 \Le l \Le i-1}(b_i-b_l)}\frac{\prod_{i+1 \Le l \Le k} (a_l-b_i)}{\prod_{i+1 \Le l \Le k}(b_l-b_i)},
        \eeqn
It is not difficult to verify that $c_1<0,\ldots, c_{k-1}<0,c_k>0.$ In view of Theorem~\ref{main-3}, it suffices to show $\sum_{i=1}^kc_i\Le 0.$ This is easy to check using the assumption and Proposition~\ref{sumc-i}
    \end{proof}
    \begin{remark}
        By Corollaries~\ref{CA-Ex}(i), \ref{c_n-pos}, it is easy to see that for any $b_k>b_{k-1},$   choice of $a_k$ lies in $[b_{k-1},b_k+\sum_{i=1}^{k-1}(b_i-a_i)]$ and the upper bound is strict, this follows from the fact that reciprocal of the completely alternating sequence $\left\{\prod_{i=1}^k\frac{(n+a_i)}{(n+b_i)}\right\}_{n \in \mathbb Z_+}$ is a completely monotone sequence and the Proposition~\ref{nec-cond} (to be seen later).
    \end{remark}
    The following corollary presents a complete characterization of certain classes of completely monotone nets in two variables. This is an application of Theorems~\ref{CAJCM}, \ref{CA-root-2} (also see \cite{ACN}[Theorems~3.1, 3.3, Question~1.1]).
    \begin{corollary}\label{bi-comp}
        Let $a_j, b_j \in (0, \infty),$ $j= 1,2.$ Then the following statements are valid$:$
\begin{enumerate}
\item[$\mathrm{(i)}$] $\left\{\frac{m+a_1}{(m+b_1)+(m+a_1)n}\right\}_{m,n \in \mathbb{Z}_+} \in \mathcal{M}(\mathbb Z_+^2)$  if and only if $b_1 \Le a_1.$
\item[$\mathrm{(ii)}$] $\left\{\frac{m+a_1}{(m+b_1)(m+b_2)+(m+a_1)n}\right\}_{m,n \in \mathbb{Z}_+} \in \mathcal{M}(\mathbb Z_+^2)$ if and only if $b_1\Le a_1\Le b_2.$
\item[$\mathrm{(iii)}$]   $\left\{\frac{(m+a_1)(m+a_2)}{(m+b_1)(m+b_2)+(m+a_1)(m+a_2)n}\right\}_{m,n \in \mathbb{Z}_+} \in \mathcal{M}(\mathbb Z_+^2)$   if and only if $$b_1\Le a_1\Le b_2, b_1+b_2 \Le a_1+a_2.$$
\end{enumerate}
    \end{corollary}
    \begin{proof}
        The proof follows from Theorem~\ref{CAJCM} and Theorem \ref{CA-root-2}.
    \end{proof}
    As an immediate application of Theorem~\ref{CAJCM}, we are able to recover Theorem 3.1 in \cite{ACN} and along with that, we obtain another equivalent statement.  
    \begin{corollary}\label{bi-poly} For $a>0,b\Ge0, c\Ge0, d\Ge0,$ the following are equivalent:
    \begin{enumerate}
        \item[$\mathrm{(i)}$]   The net $\left\{\frac{1}{a +bm+cn+dmn}\right\}_{m,n \in \mathbb{Z}_+}$ is completely monotone.
         \item[$\mathrm{(ii)}$]   The net $\left\{\frac{c+dm}{a +bm+cn+dmn}\right\}_{m,n \in \mathbb{Z}_+}$ is completely monotone.
          \item[$\mathrm{(iii)}$]  $ad-bc\Le 0.$
    \end{enumerate}
\end{corollary}
\begin{proof}
The equivalence of (ii) and (iii) follows from Theorem~\ref{CAJCM}. Note that, (ii)$\implies$(i) is immediate from the fact that the sequence $\{\frac{1}{c+dm}\}_{m \in \mathbb Z_+}$ is completely monotone and (i)$\implies$(iii) follows from Lemma 2.9 in \cite{ACN}.   
\end{proof}
     \end{section}
     \allowdisplaybreaks
     
\begin{section}{Necessary conditions}\label{sec Necessary conditions}
In this section, we study the necessary conditions for the complete monotonicity of  \eqref{Problem-2}. As a consequence, we also obtain necessary conditions for the sequences of the type  \eqref{Problem-2} which are completely alternating.

We first show that Ball's sufficiency conditions for the complete monotonicity of  \eqref{Problem-2} are not necessary and then present a proof of Theorem~\ref{main-1} with the necessary results required to prove.\\
\textbf{Example:}
    Let $p(x)=(x+1.5)(x+2)(x+4)$  and $q(x)=(x+1)(x+3)(x+3.5).$ For every $n \in \mathbb Z_+,$
\beq \label{Example-1}\notag
\frac{p(n)}{q(n)}&=& 1+\frac{3/10}{n+1}-\frac{3/2}{n+3}+\frac{12/5}{2n+7}\\
&=& \int_{[0,1]}t^n(d\delta_{1}(t)+(\frac{3}{10}-\frac{3}{2}t^2+\frac{6}{5}t^{5/2})dt).
\eeq
Let $w(t)=\frac{3}{10}-\frac{3}{2}t^2+\frac{6}{5}t^{5/2},\,  t \in (0,1).$  Note that $$w'(t)=-3t(1-\sqrt{t})<0, \quad t \in (0,1).$$
This yields that $w(\cdot)$ is decreasing on $(0,1).$ Since $w(1)=0,$ $w(t)\Ge0, t \in [0,1].$ This, together with \eqref{Example-1} shows that $\left\{\frac{p(n)}{q(n)}\right\}_{n \in \mathbb Z_+}$ is a completely monotone sequence. Equivalently, the function $\frac{p(x)}{q(x)}$ is completely monotone (see \cite[Proposition~6]{AA2001}). But this does not satisfy  Ball's sufficiency conditions.

The following lemma presents a simple characterization of the complete monotonicity of a special class. A proof in a more general setting of the following lemma has been obtained in \cite[Theorem~4.1]{RZ2019}. We present an elementary proof of the following lemma, and this will play an essential role in proving Theorem~\ref{main-1}.
\begin{lemma}\label{special-case}
    Let $k \in \mathbb N,$ and $0<b_1\Le b_2\Le\ldots\Le b_k.$  Let  $p(x)=(x+a_1), a_1>0$ and $q(x)=(x+b_1)(x+b_2)\ldots(x+b_k).$ Then, $\left\{\frac{p(n)}{q(n)}\right\}_{n \in \mathbb Z_+}$ is completely monotone  if and only if $b_1\Le a_1.$
\end{lemma}
\begin{proof}
    For a proof of the sufficiency part, assume that $b_1\Le a_1.$ It is easy to see that, $\left\{\frac{n+a_1}{n+b_1}\right\}_{n \in \mathbb Z_+}$ is a  sum of two completely monotone  sequence, hence $\left\{\frac{n+a_1}{n+b_1}\right\}_{n \in \mathbb Z_+}$ is completely monotone. Since the product of completely monotone sequences is completely monotone, the sequence $\left\{\frac{p(n)}{q(n)}\right\}_{n \in \mathbb Z_+}$ is completely monotone. For a proof of necessity part, assume that the sequence $\left\{\frac{p(n)}{q(n)}\right\}_{n \in \mathbb Z_+}$ is completely monotone. By multiplying, $\frac{x+b_i}{x+b_i-h}$ for any $h \in (0, b_i)$ if necessary, we can assume that all $b_i$'s are distinct. Let, if possible, $b_1>a_1.$ For every $n \in \mathbb Z_+,$
    \beqn
    \frac{p(n)}{q(n)}&=&\left(1+\frac{a_1-b_1}{n+b_1}\right)\frac{1}{(n+b_2)\ldots(n+b_k)}\\
    &=& \frac{1}{(n+b_2)\ldots(n+b_n)}+\frac{a_1-b_1}{(n+b_1)(n+b_2)\ldots(n+b_k)}.
    \eeqn
    By the partial fraction decomposition, there exists $\alpha_1, \ldots \alpha_n \in \mathbb R,$ such that
    \beqn
    \frac{p(n)}{q(n)}&=&(a_1-b_1)\frac{\alpha_1}{n+b_1}+\frac{\alpha_2}{n+b_2}\cdots +\frac{\alpha_k}{n+b_k}, \quad n \in \mathbb Z_+.
    \eeqn
    Observe that $\alpha_1>0.$ Note that the representing measure of $\left\{\frac{p(n)}{q(n)}\right\}_{n \in \mathbb Z_+}$ is the weighted Lebesgue measure with weight function given by 
    $$w(t)=t^{b_1-1}((a_1-b_1)\alpha_1 +\alpha_2 t^{b_2-b_1}\cdots +\alpha_k t^{b_k-b_1}), \quad t \in (0,1).$$
    It is easy to see that there exists $\epsilon>0$ such that $w(t)<0, \, t \in (0, \epsilon).$ This contradicts the assumption. Hence $b_1\Le a_1.$
\end{proof}

We recall the following result which will be crucial in proving Theorem~\ref{main-1}.
\begin{proposition}\label{nec-cond}
Let $\{a_i\}_{i=1}^k$and $\{b_i\}_{i=1}^k$ be a non-decreasing sequence of positive real numbers.  Let $p(x)=(x+a_1)\ldots(x+a_{k}), q(x)=(x+b_1)\ldots(x+b_{k}).$ Assume that $\left\{\frac{p(n)}{q(n)}\right\}_{n \in \mathbb Z_+}$ is completely monotone. Then, $\sum_{i=1}^k b_i \Le \sum_{i=1}^ka_i.$
\end{proposition}
\begin{proof}
    See \cite[Remark]{Kball}.
\end{proof}
 Now, K. Ball's sufficiency conditions combined with Propositions~\ref{nec-cond} and Lemma~\ref{special-case} provide a complete characterization of completely monotone rational functions of the type \eqref{Problem-2} in the case of $k=2.$
   \begin{proposition}\label{bi-case}
        Let $p,q$ be polynomials given by $p(x)=(x+a_1)(x+a_2) $ and $q(x)=(x+b_1)(x+b_2)$ with $0<a_1 \Le a_2,0<b_1\Le b_2$ . Then, $\left\{\frac{p(n)}{q(n)}\right\}_{n \in \mathbb Z_+}$ is completely monotone if and only if $b_1\Le a_1, b_1+b_2\Le a_1+a_2.$
\end{proposition} 
 We are now ready to prove Theorem~\ref{main-1}.
 \begin{proof}[Proof of Theorem~\ref{main-1}]
 \textit{(i)} We use the finite induction on $k$. For the base cases, take $k=1$ and $k=2$ follows from Proposition~\ref{bi-case}. Now, assume the induction hypothesis. Let if possible there exists a $i_0$ such that $a_{i_0}< b_{i_0}$ (otherwise the result is true trivially). Clearly $i_0\neq 1,$ by Lemma~\ref{special-case}. Also note that, the product of two completely monotone sequences is completely monotone. Now, by using these facts, we conclude that the sequence
    \beqn
    \prod_{i=1, i \neq i_0}^{k} \frac{(n+a_{i})}{(n+b_{i})}
     \eeqn
     is completely monotone.
   Now take $a_i^{'}=a_i, i<i_0$ and $a_i^{'}=a_{i+1}, i_0<i+1\Le k.$ Similarly
   take $b_i^{'}=b_i, i<i_0$ and $b_i^{'}=b_{i+1}, i_0<i+1\Le k.$ By induction hypothesis there exists $\tau \in S_{k-1}$ such that $\tau (1)=1$ and
   \beq \label{a-b-iq}\sum_{i=1}^{l} b_{\tau(i)}^{'}\Le \sum_{i=1}^{l} a_{\tau(i)}^{'}, \quad l \in \{1,\ldots, k-1\}.
   \eeq
   Let $\sigma\in S_k$ be such that $\sigma(i_0)=k, \sigma(i)=i, i < i_{0},  \sigma(i)=i-1,  i_{0}<i\Le k.$ Then by \eqref{a-b-iq}, we have
  \beq \label{a-b-inq}
  \sum_{i=1}^{l} b_{\sigma^{-1} \tau(i)}\Le \sum_{i=1}^{l} a_{\sigma^{-1} \tau(i)} \quad l \in \{1,\ldots, k-1\}.
  \eeq
   By taking $\tau(k)=k,$ we have $\sigma^{-1} \tau \in S_k$ which is the desired permutation. By Proposition \ref{nec-cond} and \eqref{a-b-inq}, we obtain \eqref{sigma-tau}. \\
\textit{(ii)} This follows from \textit{(i)} and the fact that if the sequence $\{a(k)\}_{k \in \mathbb Z_+}$ of positive real numbers is completely alternating then $\{1/a(k)\}_{k \in \mathbb Z_+}$ is completely monotone.
 \end{proof}
  It might be tempting to believe that the new necessary conditions for the complete monotonicity of \eqref{Problem-2} are sufficient. But we now present an example which shows that the necessary conditions obtained in Theorem~\ref{main-1} are not sufficient.
  
\noindent\textbf{Example:}
Let $p(x)=(x+6)(x+8)(x+14)$  and $q(x)=(x+5)(x+10)(x+13).$  For every $n \in \mathbb Z_+,$
\beq \label{example-2}\notag
\frac{p(n)}{q(n)}&=& 1+\frac{27/40}{n+5}-\frac{32/15}{n+10}+\frac{35/24}{n+13}\\
&=& \int_{[0,1]}t^n(d\delta_{1}(t)+(\frac{27}{40}t^4-\frac{32}{15}t^9+\frac{35}{24}t^{12})dt).
\eeq
Let $w(t)=\frac{27}{40}t^4-\frac{32}{15}t^9+\frac{35}{24}t^{12},\,  t \in (0,1).$ It is easy to verify that $w(t)\ngeq 0,$ for all $t \in (0,1).$ This together with \eqref{example-2} yields that $\left\{\frac{p(n)}{q(n)}\right\}_{n \in \mathbb Z_+}$ is not completely monotone.  
\end{section}

\noindent\textbf{Acknowledgements:} The authors are grateful to Prof. Sameer Chavan for some insightful comments/suggestions. The second named author acknowledges the Department of Mathematics, Birla Institute of Technology and Science K.K. Birla Goa Campus for warm hospitality and is grateful to the DST-INSPIRE Faculty Fellowship. The first-named author's research is supported by the DST-INSPIRE Faculty Fellowship No. DST/INSPIRE/04/2020/001250. 

{}

\end{document}